\documentclass[12pt,twoside]{amsart}
\usepackage{amsmath}
\usepackage{amsthm}
\usepackage{amsfonts}
\usepackage{amssymb}
\usepackage{latexsym}
\usepackage{mathrsfs}
\usepackage{amsmath}
\usepackage{amsthm}
\usepackage{amsfonts}
\usepackage{amssymb}
\usepackage{latexsym}
\usepackage{geometry}
\usepackage{dsfont}
\usepackage[dvips]{graphicx}
\usepackage{color}
\usepackage[all]{xy}

\date{}
\allowdisplaybreaks[4] \footskip=15pt
\renewcommand{\uppercasenonmath}[1]{}

\usepackage{graphicx,amssymb}
\usepackage[all]{xy}
\usepackage{amsmath}

\allowdisplaybreaks
\usepackage{amsthm}
\usepackage{color}

\theoremstyle{plain}
\newtheorem{theorem}{Theorem}[section]
\newtheorem{proposition}[theorem]{Proposition}
\newtheorem{lemma}[theorem]{Lemma}
\newtheorem{corollary}[theorem]{Corollary}
\theoremstyle{definition}
\newtheorem{example}[theorem]{Example}
\newtheorem{definition}[theorem]{Definition}

\theoremstyle{definition}
\newtheorem*{acknowledgement}{Acknowledgement}

\theoremstyle{remark}
\newtheorem{remark}[theorem]{Remark}

\newcommand{\pf}{\noindent\begin {proof}}
\newcommand{\epf}{\end{proof}}

\newcommand{\Ext}{\mbox{\rm Ext}}
\newcommand{\Hom}{\mbox{\rm Hom}}
\newcommand{\Tor}{\mbox{\rm Tor}}

\newcommand{\Prufer}{Pr\"{u}fer}

\newcommand{\Id}{\mathrm{Id}}

\def\ra{\rightarrow}

\def\Hom{{\rm Hom}}
\def\Ext{{\rm Ext}}
\def\Tor{{\rm Tor}}

\def\m{{\frak m}}
\def\p{{\frak p}}

\def\Nil{{\rm Nil}}

\def\NP{{\rm NP}}
\def\Z{{\rm Z}}

\def\T{{\rm T}}

\def\Spec{{\rm Spec}}
\def\Max{{\rm Max}}

\def\ZN{{\rm ZN}}
\def\Ann{{\rm Ann}}

\def\Krull{{\rm Krull}}

\begin{document}
\begin{center}
{\large  \bf Some remarks on nonnil-coherent rings and $\phi$-IF rings}

\vspace{0.5cm}  \ Wei Qi$^{a}$,\  Xiaolei Zhang$^{b}$,

{\footnotesize a.\ \ School of Mathematical Sciences, Sichuan Normal University, Chengdu 610068, China\\
b.\ \ Department of Basic Courses, Chengdu Aeronautic Polytechnic, Chengdu 610100, China\\

E-mail: zxlrghj@163.com\\}
\end{center}

\bigskip
\centerline { \bf  Abstract}
\bigskip
\leftskip10truemm \rightskip10truemm \noindent

Let $R$ be a commutative ring. If the nilpotent radical $\Nil(R)$ of $R$ is a divided prime ideal, then $R$ is called a  $\phi$-ring. In this paper,  we first distinguish the classes of  nonnil-coherent rings and $\phi$-coherent rings introduced by Bacem and Ali \cite{BA16}, and then characterize nonnil-coherent  rings in terms of $\phi$-flat modules, nonnil-injective modules and  nonnil-FP-injective modules. A $\phi$-ring $R$ is called a $\phi$-IF ring if any nonnil-injective module is $\phi$-flat. We obtain some module-theoretic characterizations of $\phi$-IF rings. Two examples are given to distinguish $\phi$-IF rings and IF $\phi$-rings.
\vbox to 0.3cm{}\\
{\it Key Words:} nonnil-coherent rings; $\phi$-IF rings; $\phi$-flat modules; nonnil-injective modules; nonnil-FP-injective modules.\\
{\it 2020 Mathematics Subject Classification:} Primary: 13A15; Secondary: 13F05.

\leftskip0truemm \rightskip0truemm
\bigskip

Throughout this paper, all rings are commutative with identity and all modules are unital.  Recall from \cite{A97} that a  ring $R$ is called an \emph{$\NP$-ring} provided that  $\Nil(R)$ is a prime ideal, and a \emph{$\ZN$-ring} provided that $\Z(R)=\Nil(R)$ where $\Z(R)$ is the set of all zero-divisors of $R$. A prime ideal $P$ of $R$ is called \emph{divided prime} if $P\subsetneq (x)$, for every $x\in R-P$.  A ring $R$ is called a \emph{$\phi$-ring}, denoted by $R\in \mathcal{H}$, if  $\Nil(R)$ is a divided prime ideal of $R$. A $\phi$-ring is called a \emph{ strongly $\phi$-ring} if it is also a $\ZN$-ring.  Recall from \cite{FA04} that for a $\phi$-ring $R$ with total quotient ring $\T(R)$, the map $\phi: T(R)\rightarrow R_{\Nil(R)}$ such that $\phi(\frac{b}{a})=\frac{b}{a}$ is a ring homomorphism, and the image of $R$, denoted by $\phi(R)$, is a strongly $\phi$-ring. The classes of $\phi$-rings and strongly $\phi$-rings are good extensions of  integral domains to commutative rings with zero-divisors. In 2002, Badawi \cite{A03}  generalized the concept of Noetherian rings to that of nonnil-Noetherian rings in which all nonnil ideals are finitely generated. They showed that a $\phi$-ring $R$ is nonnil-Noetherian if and only if $\phi(R)$ is nonnil-Noetherian, if and only if $R/\Nil(R)$ is a Noetherian domain. Generalizations of Dedekind domains, Pr\"{u}fer domains, Bezout domains, Pseudo-valuation domains and Krull domains to the context of rings that are in the class $\mathcal{H}$ are also introduced and studied (see \cite{FA04,FA05, A01,ALT06}).

In 2016, Bacem and Ali \cite{BA16} introduced two versions of coherent rings that are in the class $\mathcal{H}$.  A $\phi$-ring $R$ is called  \emph{nonnil-coherent} provided that each finitely generated nonnil ideal of $R$ is finitely presented, and $R$ is called  \emph{$\phi$-coherent} provided that  $\phi(R)$ is a nonnil-coherent ring. However, they neither showed these two versions of coherent rings coincide for all $\phi$-rings  nor presented any example to distinguish them.  In Section 1 of this article, we show that any nonnil-coherent ring is $\phi$-coherent (see Proposition \ref{non-c is phi-c}), and give an example to show the converse does not hold (see Example \ref{not non-coh}). Bacem and Ali \cite{BA16} obtained the Chase Theorem for nonnil-coherent rings using $\phi$-flat modules.  We continue to characterize nonnil-coherent rings  in terms of nonnil-injective modules and nonnil-FP-injective modules. Actually, we generalized Stenstr\"{o}m's, Dai-Ding's and Chen's results on coherent rings to these of  nonnil-coherent rings (see Theorem \ref{s-d-d-non} and Theorem \ref{phi-coh-fp}).

Recall from \cite{J73} that a ring $R$ is called an \emph{IF ring} (also is called  \emph{semi-regular} in \cite{AK17, M85}) if any injective $R$-module is flat. The class of IF rings can be seen as a natural extension of that of QF rings to coherent rings and has many module-theoretic characterizations (see \cite{FD10, MD07, M85}).  Recently, Wang and Kim \cite{WK21} introduced and studied a new version of IF rings under $w$-operations. In Section 2 of this article, we generalize the class of IF rings to that of rings in $\mathcal{H}$. A $\phi$-ring $R$ is called a \emph{$\phi$-IF ring} if any nonnil-injective $R$-module is $\phi$-flat.  We obtain that a $\phi$-ring $R$ is a $\phi$-IF ring if and only  if $R$ is  nonnil-coherent and $R_{\m}$ is a $\phi$-IF ring for any $\m\in \Max(R)$, if and only if the class of $\phi$-flat modules is equal to that of nonnil-FP-injective modules, if and only if  any $\phi$-torsion $R$-module is $\phi$-copure flat, if and only if $R/\Nil(R)$ is a field (see Theorem \ref{phi-IF-ext}). We show that IF rings are $\phi$-IF rings for all strongly $\phi$-rings (see Proposition \ref{IF-phi-if}) and give two examples to distinguish  IF $\phi$-rings and  $\phi$-IF rings (see Example \ref{not phi-IF} and Example \ref{not IF} ). Finally, we give an analogue of  Matlis' result to characterize  $\phi$-Pr\"{u}fer rings (see Corollary \ref{phi-pru-if}).

\section{on nonnil-coherent rings}

Recall that a ring $R$ is \emph{coherent} if any finitely generated ideal is finitely presented.  Bacem and Ali \cite{BA16} generalized the notion of coherent rings  to  two classes of rings in $\mathcal{H}$: nonnil-coherent rings and $\phi$-coherent rings. Let $R$ be a $\phi$-ring. Then
 \begin{enumerate}
    \item $R$ is called \emph{nonnil-coherent}, provided that any finitely generated nonnil ideal of $R$ is finitely presented.
        \item $R$ is called \emph{ $\phi$-coherent}, provided that $\phi(R)$ is nonnil-coherent.
 \end{enumerate}

Obviously, if a $\phi$-ring $R$ is coherent then $R$ is nonnil-coherent, and a strongly $\phi$-ring is nonnil-coherent if and only if it is $\phi$-coherent.

\begin{proposition}\label{nonn-coh}
Let $R$ be a $\phi$-ring, then   the following assertions are equivalent:
\begin{enumerate}
    \item  $R$ is nonnil-coherent;
      \item $(0:_Rr)$ is a finitely generated ideal for any  non-nilpotent element $r\in R$, and the intersection of two finitely generated nonnil ideals of $R$ is a finitely generated nonnil ideal of $R$;
      \item $(I:_Rb)$ is a finitely generated ideal for any non-nilpotent element $b\in R$ and any   finitely generated  ideal $I$ of $R$.
 \end{enumerate}
\end{proposition}

\begin{proof}  $(1)\Leftrightarrow (2)$:  See \cite[Theorem 2.1]{BA16}.

$(1)\Rightarrow (3)$: Let $I$ be a finitely generated  ideal of $R$ and $b$ a non-nilpotent  element in $R$. Consider the following pull-back diagram: $$\xymatrix@R=20pt@C=15pt{
 0 \ar[r]^{}  & (I:_Rb)\ar[d]_{}\ar[r]^{} &R\ar[r]^{}\ar[d]_{} & (Rb+I)/I\ar@{=}[d]^{}\ar[r]^{} &  0\\
   0 \ar[r]^{} &I \ar[r]^{} &Rb+I \ar[r]^{} & (Rb+I)/I\ar[r]^{} &  0.}$$
Since $R$ is  nonnil-coherent,  $Rb+I$ is finitely presented. Since $I$ is finitely generated, $(Rb+I)/I$ is finitely presented by \cite[Theorem 2.1.2(2)]{g}. Thus $(I:_R b)$ is finitely generated by \cite[Theorem 2.1.2(3)]{g}.

$(3)\Rightarrow (1)$: Let  $I$ be a finitely generated nonnil ideal of $R$ generated by $\{a_1,...,a_{n}\}$ where each $a_i$ is non-nilpotent. We will show $I$ is finitely presented by induction on $n$. The case $n=1$ follows from the exact sequence $0\rightarrow (0:_Ra_1)\rightarrow R\rightarrow Ra_1\rightarrow 0$. For $n\geq 2$, let $L=\langle a_1,...,a_{n-1}\rangle$. Consider the exact sequence $0 \ra (L:_R a_n)\ra R\ra (Ra_n+L)/L\ra 0$. Then $(Ra_n+L)/L=I/L$ is finitely presented by (3) and \cite[Theorem 2.1.2(2)]{g}. Consider the exact sequence $0 \rightarrow L \rightarrow I \rightarrow I/L\rightarrow 0$. Since $L$ is finitely  presented by induction and $I/L$ is finitely presented, $I$ is also finitely  presented by \cite[Theorem 2.1.2(1)]{g}.
\end{proof}

\begin{proposition}\cite[Corollary 3.1]{BA16}\label{phi-coh}
Let $R$ be a $\phi$-ring, then $R$ is  $\phi$-coherent if and only if $R/\Nil(R)$ is a coherent domain.
\end{proposition}

\begin{proposition}\label{non-c is phi-c}
A $\phi$-ring $R$ is nonnil-coherent if and only if $R$ is $\phi$-coherent and $(0:_Rr)$ is a finitely generated ideal for any  non-nilpotent element $r\in R$.
\end{proposition}
\begin{proof}  Let  $R$ be a nonnil-coherent ring and $r$ a non-nilpotent element in $R$, then $(0:_Rr)$ is  finitely generated by Proposition \ref{nonn-coh}. Let $I/\Nil(R)$ and $J/\Nil(R)$ be finitely generated non-zero ideals of $R/\Nil(R)$. By \cite[Lemma 2.4]{FA04}, $I$ and $J$ are finitely generated nonnil ideals of $R$. Thus by Proposition \ref{nonn-coh}, $I\cap J$ is a  finitely generated nonnil ideal. By \cite[Lemma 2.4]{FA04} again, $I/\Nil(R)\cap J/\Nil(R)=(I\cap J)/\Nil(R)$ is a finitely generated non-zero ideal of $R/\Nil(R)$. Then $R/\Nil(R)$ is a coherent domain by \cite[Theorem 3.7.6]{fk16}. Thus $R$ is $\phi$-coherent by Proposition \ref{phi-coh}.

On the other hand, suppose $R$ is a $\phi$-coherent ring. Then $R/\Nil(R)$ is a coherent domain.   Let $I$ and $J$ be finitely generated nonnil ideals of $R$. Then $I/\Nil(R)$ and $J/\Nil(R)$ are finitely generated non-zero ideals of $R/\Nil(R)$. Thus $(I\cap J)/\Nil(R)=I/\Nil(R)\cap J/\Nil(R)$ is a finitely generated non-zero ideal of $R/\Nil(R)$. Consequently,  $I\cap J$ is a finitely generated nonnil ideal of $R$ by \cite[Lemma 2.4]{FA04}. Since $(0:_Rr)$ is a finitely generated ideal for any  non-nilpotent element $r\in R$, $R$ is nonnil-coherent by Proposition \ref{nonn-coh}.
\end{proof}

By Proposition \ref{non-c is phi-c}, any nonnil-coherent ring is $\phi$-coherent. In order to give a  $\phi$-coherent ring that is not nonnil-coherent, we recall from \cite{H88}  the idealization construction $R(+)M$ where $M$ is an $R$-module. Let  $R(+)M=R\oplus M$ as an $R$-module, and define
\begin{enumerate}
    \item ($r,m$)+($s,n$)=($r+s,m+n$),
    \item  ($r,m$)($s,n$)=($rs,sm+rn$).
\end{enumerate}
Under this construction, $R(+)M$ becomes a commutative ring with identity $(1,0)$.

\begin{lemma}\label{ideal-con-finite}
Let $R(+)M$ be the idealization construction defined as above, $N$ an $R$-submodule of $M$.  Then $0(+)N$ is an ideal of $R(+)M$. Moreover, $0(+)N$ is a finitely generated ideal of  $R(+)M$ if and only if $N$ is finitely generated over $R$.
\end{lemma}
\begin{proof} It follows from \cite[Theorem 3.1]{DW09} that $0(+)N$ is an ideal of $R(+)M$. It is easy to verify that the $R(+)M$-ideal $0(+)N$ is generated by $\{(0,n_1),...,(0,n_t)\}$ if and only if $N$ is generated by $\{n_1,...,n_t\}$ over $R$.
\end{proof}

The following example shows that the condition ``$(0:_Rr)$ is a finitely generated ideal for any  non-nilpotent element $r\in R$'' in Proposition \ref{non-c is phi-c} cannot be removed.

\begin{example}\label{not non-coh}
Let $D$ be a coherent domain not a field, $Q$ its quotient field and $E=\bigoplus\limits_{i=1}^{\infty}Q/D$. Let $R=D(+)E$ be the idealization construction. Since $E$ is divisible and $\Nil(R)=0(+)E$ is a prime ideal, $R$ is a $\phi$-ring by \cite[Corollary 3.4]{DW09}. Note that $R/\Nil(R)\cong D$, and thus $R$ is $\phi$-coherent by Proposition \ref{phi-coh}.  Let $d$ be a non-zero non-unit element of $D$. Then $(d,0)$ is a non-nilpotent element in $R$. Then  one can easily check that $(0:_R(d,0))=0(+)\Ann_Ed$, where $\Ann_Ed=\bigoplus\limits_{i=1}^{\infty}(\frac{1}{d}+D)$. Since $\Ann_Ed$ is an infinitely generated $D$-module, $(0:_R(d,0))$ is an infinitely generated $R$-ideal by Lemma \ref{ideal-con-finite}. Thus $R$ is not nonnil-coherent by Proposition \ref{non-c is phi-c}.
\end{example}
\begin{remark}
Following from Badawi \cite{A03}, a $\phi$-ring $R$ is called a nonnil-Noetherian ring provided that each nonnil ideal of $R$ is finitely generated. They showed that a $\phi$-ring $R$ is nonnil-Noetherian  if and only if $\phi(R)$ is nonnil-Noetherian  (see \cite[Theorem 2.4]{A03}). However, Example \ref{not non-coh} shows the analogue of Badawi's result is not true for nonnil-coherent.  By Proposition \ref{phi-coh}, any nonnil-Noetherian ring is $\phi$-coherent. However, nonnil-Noetherian rings are not always nonnil-coherent. Indeed, let $D$ be a Noetherian domain not a field in Example \ref{not non-coh}, then $R$ is a nonnil-Noetherian ring but not nonnil-coherent.
\end{remark}



Let $R$ be an $\NP$-ring and $M$  an $R$-module. Set
\begin{center}
$\phi$-$tor(M)=\{x\in M|Ix=0$ for some nonnil ideal $I$ of  $R \}$.
\end{center}
An $R$-module $M$ is said to be \emph{$\phi$-torsion} (resp., \emph{$\phi$-torsion free}) provided that  $\phi$-$tor(M)=M$ (resp., $\phi$-$tor(M)=0$). 
The classes of $\phi$-torsion modules and $\phi$-torsion free modules constitute a hereditary torsion theory of finite type (see \cite{S79}). The rest of this section will give some module-theoretic characterization of nonnil-coherent rings. Recall that an $R$-module $M$ is called  \emph{$\phi$-flat} provided that $\Tor^R_1(T,M)=0$ for any  $\phi$-torsion module $T$ (see \cite{ZWT13}), and   is called  \emph{nonnil-injective} provided that $\Ext_R^1(T,M)=0$ for any  $\phi$-torsion module $T$ (see \cite{ZZ19}). Certainly, the class of $\phi$-flat modules is closed under pure sub-modules, extensions and direct limits; the class of nonnil-injective modules is closed under  direct products and extensions.

\begin{definition}
Let $R$ be an $\NP$-ring. An $R$-module $M$ is called  \emph{nonnil-FP-injective} provided that $\Ext_R^1(T,M)=0$ for  any finitely presented $\phi$-torsion module $T$.
\end{definition}

Note that the class of nonnil-FP-injective modules is closed under direct sums, direct products, extensions and pure sub-modules.

\begin{proposition}\label{flat-FP-injective}
Let $R$ be an $\NP$-ring, then the following assertions are equivalent:
\begin{enumerate}
    \item $M$ is $\phi$-flat;
      \item   $\Hom_R(M,E)$ is nonnil-injective for any injective module $E$;
     \item   $\Hom_R(M,E)$ is nonnil-FP-injective for any injective module $E$;
    \item  if $E$ is an injective cogenerator, then $\Hom_R(M,E)$ is nonnil-injective.
       \item  if $E$ is an injective cogenerator, then $\Hom_R(M,E)$ is nonnil-FP-injective.
\end{enumerate}
\end{proposition}
\begin{proof}
$(1)\Rightarrow (2)$: Let $T$ be a  $\phi$-torsion $R$-module and $E$ an injective $R$-module. Since $M$ is $\phi$-flat, $\Ext_R^1(T,\Hom_R(M,E))\cong\Hom_R(\Tor_1^R(T,M),E)=0$.  Thus $\Hom_R(M,E)$ is  nonnil-injective.

$(2)\Rightarrow (3)\Rightarrow (5)$ and $(2)\Rightarrow (4)\Rightarrow (5)$: Trivial.

$(5)\Rightarrow (1)$:  Let $I$ be a finitely generated nonnil ideal of $R$ and $E$ an injective cogenerator. Since $\Hom_R(M,E)$ is nonnil-FP-injective,  $\Hom_R(\Tor_1^R(R/I,M),E)\cong \Ext_R^1(R/I,\Hom_R(M,E))=0$. Since $E$ is an injective cogenerator, $\Tor_1^R(R/I,M)=0$. Thus $M$ is $\phi$-flat by \cite[Theorem 3.2]{ZWT13}.
\end{proof}

Bacem and Ali \cite{BA16} generalized the Chase Theorem for coherent rings to that for nonnil-coherent rings.
\begin{theorem}\cite[Theorem 2.4]{BA16} \label{non-coh-flat} Let $R$ be a $\phi$-ring. The following statements are equivalent:
\begin{enumerate}
    \item  $R$ is nonnil-coherent;
    \item any direct product of $\phi$-flat $R$-modules is $\phi$-flat;
   \item  for any indexing set $I$, any $R$-module $R^I$ is $\phi$-flat.
\end{enumerate}
\end{theorem}

\begin{lemma}\label{sub-fp} Let $R$ be a nonnil-coherent ring. Let $T$ be a finitely presented $\phi$-torsion module  generated by  $\{t_1,...,t_{k},t_{k+1}\} $ with $k\geq 1$ and $T_k$ the submodule of $T$ generated by $\{t_1,...,t_{k}\}$. Then $T_k$ is finitely presented.
\end{lemma}
\begin{proof}
Note $T/T_k=(T_k+Rt_k)/T_k\cong Rt_k/(T_k\cap Rt_k)\cong R/I$ where $I=(0:_Rt_k+T_k\cap Rt_k)$ is an ideal of $R$. Since $T$ is finitely presented $\phi$-torsion and  $T_k$ is finitely generated, then $I$
 is a finitely generated nonnil ideal of $R$ by \cite[Theorem 2.1.2]{g}. Since  $R$ is nonnil-coherent, then $I$ is finitely presented. Consider the following Pull-back diagram:
$$\xymatrix@R=20pt@C=25pt{ &  0\ar[d]&0\ar[d]&&\\
&  K\ar[d]\ar@{=}[r]&K\ar[d]&&\\
0 \ar[r]^{} & X\ar[d]\ar[r]^{} &F \ar[d]\ar[r]^{} &R/I\ar@{=}[d]\ar[r]^{} &  0\\
0 \ar[r]^{} & T_k\ar[d]\ar[r]^{} & T \ar[d]\ar[r]^{} &R/I\ar[r]^{} &  0\\
& 0 &0 &&\\}$$
where $F$  is finitely generated free. Then $K$ is finitely generated by \cite[Theorem 2.1.2(3)]{g}. Since $I$ is finitely presented, $X$ is finitely presented by \cite[Theorem 2.1.2(3)]{g} again.  Thus $T_k$ is finitely presented by \cite[Theorem 2.1.2(2)]{g}.
\end{proof}

In 1970, Stenstr\"{o}m \cite{S70} obtained that a ring $R$ is coherent if and only if any direct limit of FP-injective modules is FP-injective. In 2008, Pinzon \cite{p08} showed that if $R$ is coherent, the class of FP-injective modules is (pre)covering. Recently, Dai and Ding \cite{DD18,DD19} showed that the converse of Pinzon's result also hold. The next result generalizes these results to nonnil-coherent rings.

\begin{theorem}\label{s-d-d-non}
Let $R$ be a $\phi$-ring. The following statements are equivalent:
\begin{enumerate}
    \item  $R$ is nonnil-coherent.
    \item The class of nonnil-FP-injective $R$-modules is closed under pure quotients.
   \item  The class of nonnil-FP-injective $R$-modules is closed under direct limits.
    \item  The class of nonnil-FP-injective $R$-modules is precovering.
    \item  The class of nonnil-FP-injective $R$-modules is covering.
\end{enumerate}
\end{theorem}
\begin{proof} $(1)\Rightarrow (3)$: Let $\{M_i\}_{i\in\Gamma}$ be a direct system of nonnil-FP-injective $R$-modules. Let $T$ be a finitely presented $\phi$-torsion module  generated by $n$ elements $\{t_1,...,t_n\}$. We will prove $\Ext_R^1(T,\lim\limits_{\longrightarrow} M_i)=0$ by induction on $n$. Denote $M=\lim\limits_{\longrightarrow} M_i$. If $n=1$, $T=Rt_1$. Then there exists an exact sequence $0\rightarrow (0:_Rt_1)\rightarrow R\rightarrow Rt_1\rightarrow 0$. Since $T$ is finitely presented $\phi$-torsion, then $(0:_Rt_1)$ is a finitely generated nonnil-ideal of $R$. Consider the following commutative diagrams with exact rows.

$$\xymatrix@R=20pt@C=20pt{
     \lim\limits_{\longrightarrow } \Hom_R(R,M_i) \ar[d]_{\varphi_R}\ar[r]^{} & \lim\limits_{\longrightarrow }\Hom_R((0:_Rt_1),M_i)  \ar[r]^{}\ar[d]^{\varphi_{(0:_Rt_1)}}& \lim\limits_{\longrightarrow }\Ext^1_R(Rt_1,M_i) \ar[r]^{}\ar[d]^{\varphi^1_{Rt_1}}&0 \\
    \Hom_R(R,\lim\limits_{\longrightarrow } M_i)  \ar[r]^{} &\Hom_R((0:_Rt_1),\lim\limits_{\longrightarrow } M_i)\ar[r]^{} & \Ext^1_R(Rt_1,\lim\limits_{\longrightarrow } M_i) \ar[r]^{} &0.
}$$
Since $R$ is nonnil-coherent, $(0:_Rt_1)$ is finitely presented. Then  $\varphi_{(0:_Rt_1)}$ is an isomorphism, and thus $\varphi^1_{Rt_1}$ is an isomorphism as $\varphi_R$  is an isomorphism. Consequently,  $\Ext_R^1(Rt_1,\lim\limits_{\longrightarrow} M_i)=0$.  Let $T$ be a finitely presented $\phi$-torsion module  generated by  $\{t_1,...,t_{k},t_{k+1}\}$ and $T_k$ the $\phi$-torsion submodule of $T$ generated by $\{t_1,...,t_{k}\}$.  Then $T_k$ is finitely presented $\phi$-torsion by Lemma \ref{sub-fp}. Consider the exact sequence $0\rightarrow T_k\rightarrow T\rightarrow R/I\rightarrow 0$ where $I=(0:_Rt_k+T_k\cap Rt_k)$. By the proof of Lemma \ref{sub-fp}, $R/I$ is finitely presented $\phi$-torsion. We have a long exact sequence $$\Ext_R^1(R/I,\lim\limits_{\longrightarrow} M_i)\rightarrow \Ext_R^1(T,\lim\limits_{\longrightarrow} M_i)\rightarrow \Ext_R^1(T_k,\lim\limits_{\longrightarrow} M_i).$$
By induction, $\Ext_R^1(T_k,\lim\limits_{\longrightarrow} M_i)=\Ext_R^1(R/I,\lim\limits_{\longrightarrow} M_i)=0$, thus $\Ext_R^1(T,\lim\limits_{\longrightarrow} M_i)=0$. Consequently, $\lim\limits_{\longrightarrow} M_i$ is nonnil-FP-injective.

$(3)\Rightarrow (1)$: Let $I$ be a finitely generated nonnil ideal,  $\{M_i\}_{i\in\Gamma}$  a direct system of $R$-modules. Let $\alpha: I\rightarrow \lim\limits_{\longrightarrow }M_i$ be a homomorphism. For any $i\in \Gamma$, $E(M_i)$ is the injective envelope of $M_i$. Then $E(M_i)$ is nonnil-FP-injective. By (3), there exists an $R$-homomorphism $\beta:R\rightarrow \lim\limits_{\longrightarrow }E(M_i)$ such that the following  diagram commutes:
$$\xymatrix@R=20pt@C=25pt{
0\ar[r]^{}&I\ar[r]^{}\ar[d]_{\alpha}&  R\ar[r]^{}\ar@{.>}[d]^{\beta} & R/I\ar@{.>}[d]^{}\ar[r]^{} &0\\
0\ar[r]^{}&{\lim\limits_{\longrightarrow }}M_i \ar[r]^{}&{\lim\limits_{\longrightarrow }}E(M_i)  \ar[r]^{} &{\lim\limits_{\longrightarrow }}E(M_i)/M_i\ar[r]^{} & 0.\\}$$
Thus, by \cite[Lemma 2.13]{gt}, there exists $j\in \Gamma$, such that $\beta$ can factor through $R\xrightarrow{\beta_j} E(M_j)$. Consider the following commutative diagram:
$$\xymatrix@R=20pt@C=25pt{
0\ar[r]^{}&I\ar[r]^{}\ar@{.>}[d]_{\alpha_j}&  R\ar[r]^{}\ar[d]^{\beta_j} & R/I\ar@{.>}[d]^{}\ar[r]^{} &0\\
0\ar[r]^{}&M_j \ar[r]^{}&E(M_j)  \ar[r]^{} &E(M_j)/M_j\ar[r]^{} & 0.\\}$$
Since the composition $I\rightarrow R\rightarrow E(M_j)\rightarrow E(M_j)/M_j$ becomes to be $0$ in the direct limit, we can assume $I\rightarrow R\rightarrow E(M_j)$ can factor through some  $I\xrightarrow{\alpha_j} M_j$. Thus $\alpha$ can factor through $M_j$. Consequently, the natural homomorphism  $\lim\limits_{\longrightarrow } \Hom_R(I,M_i)\xrightarrow{\phi}  \Hom_R(I, \lim\limits_{\longrightarrow }M_i)$ is an epimorphism. Now suppose  $\{M_i\}_{i\in\Gamma}$ is a direct system of finitely generated $R$-modules such that $\lim\limits_{\longrightarrow } M_i=I$. Then there exists $f\in \Hom_R(I,M_j)$ with $j\in \Gamma$ such that the identity map  $\Id_I= \phi(u_j(f))$ where $u_j$ is the natural homomorphism $\Hom_R(I,M_j)\rightarrow \lim\limits_{\longrightarrow } \Hom_R(I,M_i)$. Then $I$ is a direct summand of $M_j$, and thus $I$ is finitely generated. It follows from \cite[Section 24.9, Section 24.10]{w} that  $I$ is finitely presented.

$(2)\Leftrightarrow (3)\Leftrightarrow (4)\Leftrightarrow (5)$: Follows from \cite[Lemma 3.4]{ZWQ20}.
\end{proof}

In 1993, Chen and Ding in \cite{c93} showed that a ring $R$ is coherent if and only if $\Hom_R(M,E)$ is flat for any absolutely pure $R$-module $M$ and injective $R$-module $E$ if and only if $\Hom_R(M,E)$ is flat for any injective $R$-modules $M$ and $E$. We also generalize this result to nonnil-coherent rings.

\begin{theorem}\label{phi-coh-fp}
 Let $R$ be a $\phi$-ring. The following statements are equivalent:
\begin{enumerate}
    \item $R$ is nonnil-coherent;
    \item   $\Hom_R(M,E)$ is $\phi$-flat for any  nonnil-FP-injective module $M$  and any injective module $E$;
   \item  $\Hom_R(M,E)$ is $\phi$-flat for any  nonnil-injective module $M$  and any injective module $E$;
       \item$\Hom_R(\Hom_R(M,E_1),E_2)$ is $\phi$-flat for any $\phi$-flat module $M$  and any injective modules $E_1,\ E_2$;
         \item if $E_1$ and $E_2$ are injective cogenerators, then $\Hom_R(\Hom_R(M,E_1),E_2)$ is $\phi$-flat for any $\phi$-flat module $M$.

\end{enumerate}
\end{theorem}
\begin{proof}
$(2)\Rightarrow (3)$ and $(4)\Rightarrow (5)$: Trivial.

$(3)\Leftrightarrow (4)$: Follows from Proposition \ref{flat-FP-injective}.

 $(1)\Rightarrow (2)$: Let $I$ be a finitely generated nonnil ideal of $R$.  Consider the following commutative diagram with exact rows($(-,-)$ is instead of $\Hom_R(-,-)$):
$$\xymatrix@R=20pt@C=15pt{
0\ar[r]^{}&\Tor_1^R((M,E),R/I) \ar[r]^{}\ar[d]^{\psi^1_{R/I}} &(M,E)\otimes_R I \ar[d]_{\psi_{I}}\ar[r]^{} & (M,E)\otimes_R R \ar[d]_{\psi_{R}}^{\cong}\ar[r]^{} &(M,E)\otimes_R R/I \ar[d]_{\psi_{R/I}}^{}\ar[r]^{} &0\\
0\ar[r]&(\Ext_R^1(R/I,M),E) \ar[r]^{} &((I,M),E)\ar[r]^{} & ((R,M),E) \ar[r]^{} &((R/I,M),E)\ar[r]^{} &0.\\ }$$
Since $I$ and $R$ are finitely presented, then $\psi_{I}$ and $\psi_{R}$ are isomorphisms by  \cite[Proposition 8.14(1)]{hh} and \cite[Theorem 2]{ELMUT1969}. Thus $\psi^1_{R/I}$ is an isomorphism by the Five Lemma.
Since $M$ is nonnil-FP-injective, $\Ext_R^1(R/I,M)=0$. Then $\Tor_1^R(\Hom_R(M,E),R/I)=0$, and thus $\Hom_R(M,E)$ is $\phi$-flat.

$(5)\Rightarrow (1)$: Let  $\{M_i\}_{i\in\Gamma}$ be a family of $\phi$-flat modules and $E_1$ and $E_2$ be injective cogenerators. Then $\bigoplus\limits_{i\in\Gamma}M_i$ is $\phi$-flat. Then
 $$\Hom_R(\Hom_R(\bigoplus\limits_{i\in\Gamma}M_i,E_1),E_2)\cong \Hom_R(\prod\limits_{i\in\Gamma}\Hom_R(M_i,E_1),E_2)$$ is $\phi$-flat by (5).
Note that $\bigoplus\limits_{i\in\Gamma}\Hom_R(M_i,E_1)$  is the pure submodule of $\prod\limits_{i\in\Gamma}\Hom_R(M_i,E_1)$ by \cite[Lemma 1(1)]{CS81}.
Thus the natural epimorphism $$\Hom_R(\prod\limits_{i\in\Gamma}\Hom_R(M_i,E_1),E_2)\twoheadrightarrow \Hom_R(\bigoplus\limits_{i\in\Gamma}\Hom_R(M_i,E_1),E_2)$$ splits by \cite[Lemma 2.19]{gt}.
It follows that $$ \prod\limits_{i\in\Gamma}\Hom_R(\Hom_R(M_i,E_1),E_2) \cong \Hom_R(\bigoplus\limits_{i\in\Gamma}\Hom_R(M_i,E_1),E_2)$$ is $\phi$-flat. By \cite[Corollary 2.21]{gt}, $\prod\limits_{i\in\Gamma} M_i$ is the pure submodule of $\prod\limits_{i\in\Gamma}\Hom_R(\Hom_R(M_i,E_1),E_2)$. Thus $\prod\limits_{i\in\Gamma} M_i$ is $\phi$-flat. By Theorem \ref{non-coh-flat}, $R$ is nonnil-coherent.
\end{proof}

\section{on $\phi$-IF rings}

Recall from \cite{J73} that a ring $R$ is called an \emph{IF ring} if any injective $R$-module is flat. We generalize the concept of IF rings to that of rings  in $\mathcal{H}$ using nonnil-injective modules and $\phi$-flat modules.

\begin{definition}\label{def-phi-IF}
Let $R$ be a $\phi$-ring. Then $R$ is said to be a \emph{$\phi$-IF ring} provided that any nonnil-injective $R$-module is $\phi$-flat.
\end{definition}

\begin{lemma}\label{phi-IF-coh}
Let $R$ be a $\phi$-IF ring, then $R$ is  nonnil-coherent.
\end{lemma}

\begin{proof}
Let $M$ be a nonnnil-injective  $R$-module and $E$ an injective $R$-module. Then $M$ is  $\phi$-flat as $R$ is a $\phi$-IF ring. Thus $\Hom_R(M,E)$ is nonnnil-injective by Proposition \ref{flat-FP-injective}. Consequently, $R$ is nonnil-coherent by Theorem \ref{phi-coh-fp}.
\end{proof}

\begin{proposition}\label{phi-iso}
Let $R$ be a nonnil-coherent ring, $M$ an $R$-module. Suppose  $T$ is a finitely presented $\phi$-torsion module and  $E$ is an injective  $R$-module. Then $$\Tor_1^R(T,\Hom_R(M,E))\cong \Hom_R(\Ext_R^1(T,M),E).$$
\end{proposition}

\begin{proof} Suppose $T$ is generated by $n$ elements, then there is a exact sequence $0\rightarrow K\rightarrow P\rightarrow T\rightarrow 0$ with $P=R^n$ and $K$ finitely generated. We will show $K$ is finitely presented by induction on $n$. If $n=1$, then $K$ is a finitely generated nonnil ideal of $R$. Thus $K$ is finitely presented as $R$ is nonnil-coherent. Suppose $n=k+1$, then there is commutative diagram:
$$\xymatrix@R=20pt@C=25pt{
0 \ar[r]^{} & K\cap R^{k} \ar@{^{(}->}[d]\ar[r]^{} &R^{k} \ar@{^{(}->}[d]\ar[r]^{} &R^{k}/K\cap R^{k}  \ar@{^{(}->}[d]\ar[r]^{} &  0\\
0 \ar[r]^{} & K\ar@{->>}[d]\ar[r]^{} & R^{k+1} \ar@{->>}[d]\ar[r]^{} &R^{k+1}/K\ar@{->>}[d]\ar[r]^{} &  0\\
0 \ar[r]^{} & I \ar[r]^{} & R\ar[r]^{} &R/I\ar[r]^{} &  0,\\}$$
where $I=K/K\cap R^{k}$ is an ideal of $R$. Since $T=R^{k+1}/K$ is finitely presented $\phi$-torsion, then  $R^{k}/K\cap R^{k}$ is finitely presented $\phi$-torsion by Lemma \ref{sub-fp}. Thus $R/I$ is finitely presented $\phi$-torsion by \cite[Theorem 2.1.2(2)]{g}. Since $R^{k}/K\cap R^{k}$ is generated by $k$ elements,  $K\cap R^{k}$ and $I$ are  finitely presented  by induction. Thus $K$ is  finitely presented by \cite[Theorem 2.1.2(1)]{g}.

Let $F$ be a  $\phi$-flat $R$-module and $E$ a injective $R$-module. Then there is a  commutative diagram with exact rows($(-,-)$ is instead of $\Hom_R(-,-)$):
 $$\xymatrix@R=20pt@C=15pt{
0\ar[r]^{}&\Tor_1^R(T,(F,E)) \ar[r]^{}\ar[d]^{\psi^1_{T}} &(F,E)\otimes_R K \ar[d]_{\psi_{K}}^{\cong}\ar[r]^{} & (F,E)\otimes_R P \ar[d]_{\psi_{P}}^{\cong}\ar[r]^{} &(F,E)\otimes_R T \ar[d]_{\psi_{T}}^{\cong}\ar[r]^{} &0\\
0\ar[r]&((\Ext_R^1(T,F),E) \ar[r]^{} &((K,F),E)\ar[r]^{} & ((P,M),E) \ar[r]^{} &((T,M),E)\ar[r]^{} &0.\\ }$$
Since $K, F$ and $T$ are finitely presented, $\psi_{K}, \psi_{F}$ and $\psi_{T}$ are isomorphisms by \cite[Theorem 2.6.13(2)]{fk16}. Thus $\psi^1_{T}$ is isomorphism by the Five Lemma.
\end{proof}

Let $M$ be an $R$-module and  $N$ a submodule of $M$. Then  $N\hookrightarrow M$ is said to be a $\phi$-embedding map provided that $M/N$ is a $\phi$-torsion module. Let $M$ be an $R$-module, then there is a nonnil-injective envelope, denoted by $E_{\phi}(M)$,  of $M$ (see  \cite[Theorem 2.7]{ZZ19}). Note that $M\hookrightarrow E_{\phi}(M)$ is a $\phi$-embedding map (see \cite[Theorem 2.14]{ZZ19}).

Recall from \cite{MD07} that an $R$-module $M$ is said to be \emph{copure flat} if $\Tor^R_1(E,M)=0$ for any injective module $E$. It was proved that a ring $R$ is an IF ring if and only if any $R$-module is copure flat (see \cite[Proposition 2.14]{MD07}). The following concepts give a ``strong'' version of copure flat modules.

\begin{definition} Let $R$ be an $\NP$-ring. An $R$-module $M$ is called  \emph{$\phi$-copure flat} provided that $\Tor^R_1(E,M)=0$ for any  nonnil-injective module $E$.
\end{definition}

\begin{theorem} \label{phi-IF-ext}
Let $R$ be a $\phi$-ring. The following statements are equivalent:
\begin{enumerate}
\item $R$ is a $\phi$-IF ring;
\item  $R$ is a nonnil-coherent ring and $R_{\p}$ is a $\phi$-IF ring for any $\p\in \Spec(R)$;
\item  $R$ is a nonnil-coherent ring and $R_{\m}$ is a $\phi$-IF ring for any $\m\in \Max(R)$;
\item any $R$-module can be $\phi$-embedded into a  $\phi$-flat module.
     \item any nonnil-FP-injective module is $\phi$-flat;
    \item $R$ is nonnil-coherent and any $\phi$-flat module is nonnil-FP-injective;
    \item an $R$-module $M$ is $\phi$-flat if and only if $M$ is nonnil-FP-injective;
        \item any $\phi$-torsion $R$-module is $\phi$-copure flat;
         \item $R/\Nil(R)$ is a field.
\end{enumerate}
\end{theorem}
\begin{proof}

$(1)\Rightarrow (4)$: Let $M$ be an $R$-module, $M\hookrightarrow E_{\phi}(M)$ the  $\phi$-embedding of $M$ into  $E_{\phi}(M)$.  Since $E_{\phi}(M)$ is a $\phi$-flat module by $(1)$ , then $(4)$ holds naturally.

$(4)\Rightarrow (1)$: Let $M$ be a nonnnil-injective and $M\subseteq F$ be the $\phi$-embedding of $M$ into a  $\phi$-flat module $F$. Since $F/M$ is $\phi$-torsion, $\Ext_R^1(F/M, M)=0$. Then $M$ is a direct summand of the $\phi$-flat module $F$, and thus $M$ is $\phi$-flat.

$(1)+(4)\Rightarrow (2)$:  By Lemma \ref{phi-IF-coh}, $R$ is a nonnil-coherent ring. Let $\p$ be a prime ideal of $R$, $A$ an $R_\p$-module and $F$ a $\phi$-flat $R$-module containing $A$ such that $F/A$ is $\phi$-torsion. Then $A\cong A_\p\subseteq F_\p$. Note that $F_\p/A_\p$ is a $\phi$-torsion $R_\p$-module (see  \cite[Proposition 2.12]{ZxlQ21}) and $F_\p$ is a $\phi$-flat $R_\p$-module (see \cite[Theorem 3.5]{ZWT13}). Thus  $R_{\p}$ is a $\phi$-IF ring by (4).

$(2)\Rightarrow (3)$:  Trivial.

$(1)\Rightarrow (6)$: Let $R$ be a  $\phi$-IF ring, then $R$ is nonnil-coherent  by Lemma \ref{phi-IF-coh}. Suppose  $T$ is a finitely presented $\phi$-torsion module. Let $F$ be a $\phi$-flat $R$-module and $E$ an injective $R$-module. By  Proposition \ref{phi-iso},
$$ \Tor_1^R(T,\Hom_R(F,E))\cong \Hom_R(\Ext_R^1(T,F),E).$$
By Proposition \ref{flat-FP-injective}, $\Hom_R(F,E)$ is  nonnil-injective. Since  $R$ is a  $\phi$-IF ring, $\Hom_R(F,E)$ is $\phi$-flat. Since  $T$ is finitely presented $\phi$-torsion, then   $\Tor_1^R(T,\Hom_R(F,E))=0$, and thus $\Hom_R(\Ext_R^1(T,F),E)=0$. It follows that $\Ext_R^1(T,F)=0$. Consequently, $F$ is nonnil-FP-injective.

$(6)\Rightarrow (7)$: Let $M$ be a nonnil-FP-injective $R$-module and $E$ an injective cogenerator over $R$. Let $T$ be a finitely presented $\phi$-torsion module, then $$\Hom(\Tor_1^R(T,M),E)\cong \Ext_R^1(T,\Hom_R(M,E)).$$
 Since $R$ is nonnil-coherent, $\Hom_R(M,E)$ is $\phi$-flat by Theorem \ref{phi-coh-fp}. Then $\Hom_R(M,E)$ is  nonnil-FP-injective by $(6)$. It follows that $\Ext_R^1(T,\Hom_R(M,E))=0$, and thus $\Hom(\Tor_1^R(T,M),E)$. Consequently,  $\Tor_1^R(T,M)=0.$ So $M$ is $\phi$-flat.

$(7)\Rightarrow (1)$: Note that nonnil-injective modules are nonnil-FP-injective. Thus any nonnil-injective module is $\phi$-flat by $(7)$.

 $(1)\Rightarrow (8)$: Let $M$ be a nonnil-injective $R$-module and $T$ a $\phi$-torsion module. Then $M$ is $\phi$-flat, and thus $\Tor_1^R(T,M)=0$. Consequently, $T$ is $\phi$-copure flat.

 $(8)\Rightarrow (1)$: Let  $M$ be a nonnil-injective $R$-module and $T$ a $\phi$-torsion module. Since $T$ is $\phi$-copure flat, then $\Tor_1^R(T,M)=0$. Thus $M$ is $\phi$-flat.

$(3)+(1)\Rightarrow (6)$: Let $M$ be a $\phi$-flat module, then $M_{\m}$ is a $\phi$-flat $R_{\m}$-module for any maximal ideal $\m$ of $R$. Thus $M_{\m}$ is a nonnil-FP-injective $R_{\m}$-module by  $(1)\Rightarrow (6)$ for the $\phi$-IF ring $R_{\m}$. Let $T$  be a finitely presented $\phi$-torsion module. Then, by \cite[Theorem 2.6.16]{fk16}, $\Ext_R^1(T,M)_\m\cong \Ext_{R_\m}^1(T_\m,M_\m)=0$ as $T_\m$ is a finitely presented $\phi$-torsion $R_{\m}$-module by \cite[Proposition 2.12]{ZxlQ21}. Thus $\Ext_R^1(T,M)=0$. So $M$ is nonnil-FP-injective.

$(7)\Rightarrow (5)\Rightarrow (1)$: Trivial.

$(9)\Rightarrow (1)$: If  $R/\Nil(R)$ is a field, then any $R$-module is $\phi$-flat and nonnil-injective by \cite[Theorem 1.7]{ZxlQ21}.   Thus $R$ is a $\phi$-IF ring.

$(1)\Rightarrow (9)$: Suppose $R$ is a $\phi$-IF ring. Let $E$ be an injective $R/\Nil(R)$-module, then  $E$ is a nonnil-injective $R$-module by \cite[Proposition 1.4]{ZxlQ21}. Thus $E$ is $\phi$-flat over $R$. It follows from \cite[Proposition 1.7]{ZxlZ20} that $E$ is flat over $R/\Nil(R)$. Consequently, $R/\Nil(R)$ is an IF domain, and thus is a field by \cite[Proposition 3.1]{J73}.
\end{proof}

\begin{proposition}\label{IF-phi-if}
Let $R$ be a strongly $\phi$-ring. If $R$ is an IF ring, then $R$ is a $\phi$-IF ring.
\end{proposition}
\begin{proof} Since $R$ is a strongly $\phi$-ring, any non-nilpotent element is regular. Because $R$ is an IF ring , any regular element is invertible by \cite[Proposition 2.1(1)]{M85}. Then the Krull dimension of $R$ is $0$, and thus $R/\Nil(R)$ is a field. Consequently,  $R$ is a $\phi$-IF ring by Theorem \ref{phi-IF-ext}.
\end{proof}

Note that  every IF ring is not $\phi$-IF. For example, let $R$ be a von Neumann regular ring not a field. Then $R$ is an IF ring. Since $\Nil(R)=0$ is not a prime ideal, then $R$ is not a $\phi$-ring, and thus not a $\phi$-IF ring. The following example shows that  IF rings are also not necessary $\phi$-IF rings for $\phi$-rings.

\begin{example}\label{not phi-IF}
Let $D$ be a Pr\"{u}fer domain  not a field, and $Q$ its quotient field. Let $R=D(+)Q/D$ be the idealization construction.  Then $\Nil(R)=0(+)Q/D$ is a prime ideal of $R$. Thus  $R$ is a $\phi$-ring by \cite[Corollary 3.4]{DW09}. By \cite[Example 2.12(1)]{AK17} $R$ is an IF ring. However $R/\Nil(R)\cong D$ is not a field. Thus $R$ is not a $\phi$-IF ring by Theorem \ref{phi-IF-ext}.
\end{example}

It is also showed that $\phi$-IF rings are not necessary  IF rings.

\begin{example}\label{not IF}
Let $K$ be a field  and $V=\prod\limits_{i=1}^{\infty}K$ an infinite dimensional vector space over $K$. Let $R=K(+)V$ be the idealization construction. Obviously, $R$ is a $\phi$-ring. Note that $\Nil(R)=0(+)V$. Since $K\cong R/\Nil(R)$ is a field, $R$ is a  $\phi$-IF ring by Theorem \ref{phi-IF-ext}.  Let $v$ be a vector in $V$ with each component equal to $1$. Then $(0:_R(0,v))=0(+)V$. Since $V$ is an infinite dimensional $K$-vector space, $0(+)V$ is not a finitely generated $R$-ideal by Lemma \ref{ideal-con-finite}. Then $R$ is not coherent by \cite[Theorem 2.3.2]{g}. Thus $R$ is not an IF ring by \cite[Proposition 3.3]{M85}.
\end{example}

Recall from \cite{FA04} that a $\phi$-ring $R$ is said to be a $\phi$-Pr\"{u}fer ring if  any finitely generated nonnil ideal of $R$ is $\phi$-invertible. They showed that a  $\phi$-ring $R$  is a $\phi$-Pr\"{u}fer ring if and only if $\phi(R)$ is a Pr\"{u}fer ring, if and only if $R/\Nil(R)$ is a Pr\"{u}fer domain(see \cite[Theorem 2.2, Theorem 2.6]{FA04}). Matlis \cite{M85} obtained that an integral domain $R$ is a  Pr\"{u}fer domain if and only if $R/I$ is an IF ring for any non-zero finitely generated $I$ of $R$. Now we give an analogue of  Matlis' result for $\phi$-rings.

\begin{corollary}\label{phi-pru-if}
Let $R$ be a $\phi$-ring. Then  $R$ is a  \emph{$\phi$-Pr\"{u}fer ring} if and only if $R/I$ is an IF ring for any finitely generated nonnil ideal $I$ of $R$.
\end{corollary}
\begin{proof} Let  $R$ be a  $\phi$-Pr\"{u}fer ring and $I$ a finitely generated nonnil ideal of $R$. Then  $R/\Nil(R)$ is a Pr\"{u}fer domain. Since  $I/\Nil(R)$ is a finitely generated non-zero ideal of $R/\Nil(R)$ by \cite[Lemma 2.4]{FA04}, we have $R/I=\frac{R/\Nil(R)}{I/\Nil(R)}$ is an IF ring.

Let $I/\Nil(R)$ be a finitely generated nonzero ideal of $R/\Nil(R)$. Then $I$ is a finitely generated nonnil ideal of $R$  by \cite[Lemma 2.4]{FA04} again. Thus $\frac{R/\Nil(R)}{I/\Nil(R)}=R/I$ is an IF ring. Consequently, $R/\Nil(R)$ is a Pr\"{u}fer domain. It follows that $R$ is a  $\phi$-Pr\"{u}fer ring.
\end{proof}

\begin{acknowledgement}\quad\\
The second author was supported by the Natural Science Foundation of Chengdu Aeronautic Polytechnic (No. 062026) and the National Natural Science Foundation of China (No. 12061001).
\end{acknowledgement}


\bigskip
\begin{thebibliography}{99}

\bibitem{AK17} K. Adarbeh and S. Kabbaj, Matlis semi-regular in trivial ring extensions issued from integral domains, Colloq. Math. \textbf{150} (2017), no. 2, 229-241.


\bibitem{FA04} D. F. Anderson, A. Badawi, {\it On $\phi$-Pr\"{u}fer rings and $\phi$-Bezout rings},  Houston J. Math.  \textbf{30} (2004), 331-343.

\bibitem{FA05} D. F. Anderson, A. Badawi,  {\it  On $\phi$-Dedekind rings and $\phi$-\Krull\ rings},  Houston J. Math.  \textbf{31}  (2005), 1007-1022.

\bibitem{DW09}  D. D. Anderson, M.  Winders, {\it Idealization of a module},  J. Commut. Algebra  \textbf{1} (2009), 3-56.

\bibitem{hh} L. Angeleri H\"ugel,  D. Herbera,  {\it Mittag-Leffler conditions on modules},  Indiana Univ. Math. J.  {\bf 57} (2008),2459-2517.




\bibitem{A97} A. Badawi,  {\it On divided commutative rings},  Commun. Algebra  \textbf{27} (1999), 1465-1474.

\bibitem{A01} A. Badawi, {\it On $\phi$-chained rings and $\phi$-pseudo-valuation rings},  Houston J. Math.  \textbf{27} (2001), 725-736.

\bibitem{A03}  A. Badawi, {\it On Nonnil-Noetherian rings}, Commun. Algebra \textbf{31} (2003), no. 4, 1669-1677.

\bibitem{ALT06}  A. Badawi,  T.  Lucas,  {\it On $\phi$-Mori rings},  Houston J. Math.   \textbf{32} (2006), 1-32.

\bibitem{BA16} K. Bacem, B. Ali, {\it Nonnil-coherent rings}, Beitr. Algebra Geom.  \textbf{57} (2016), no. 2, 297-305.


\bibitem{CS81} T. J. Cheatham,  D. R. Stone, {\it Flat and projective character modules}, Proc. Amer. Math. Soc.  \textbf{81} (1981), no. 2, 175-175

\bibitem{c93} J. L. Chen, Ding, N. Q. {\it The weak global dimension of commutative coherent rings}, Commun. Algebra {\bf 21}  (1993), no. 10, 3521-3528.

\bibitem{DD18} G. C. Dai,  N. Q. Ding,  {\it Coherent rings and absolutely pure covers},   Commun. Algebra   {\bf 46} (2018), no. 3, 1267-1271.

\bibitem{DD19} G. C. Dai,  N. Q. Ding, {\it Coherent rings and absolutely pure precovers},   Commun. Algebra   {\bf 47} (2019), no. 11, 4743-4748.

\bibitem{FD10} X. H. Fu, N. Q. Ding, {\it On strongly copure flat modules and copure flat dimensions}, Commun. Algebra {\bf 38} (2010), 4531-4544.

\bibitem{FS01} L. Fuchs, L. Salce, {\it Modules over non-Noetherian domains}, Mathematical Surveys and Monographs
84, AMS, 2001.

\bibitem{g}  S. Glaz,  {\it  Commutative Coherent Rings}, Lecture Notes in Mathematics, vol.  {\bf 1371}, Berlin: Spring-Verlag, 1989.

\bibitem{gt}  R. Gobel, J. Trlifaj,   {\it Approximations and endomorphism algebras of modules}, De Gruyter Exp. Math., vol.  {\bf 41}, Berlin: Walter de Gruyter GmbH \& Co. KG, 2012.

\bibitem{H88} J. A. Huckaba,  {\it Commutative Rings with Zero Divisors}. Monographs and Textbooks in Pure and Applied Mathematics, {\bf 117}, Marcel Dekker, Inc., New York, 1988.

\bibitem{J73}  S. Jain, {\it Flat and FP-injectivity},  Proc. Amer. Math. Soc.  {\bf 41} (1973), 437-442.





\bibitem{ELMUT1969} E. Lenzing,  {\it Endlich pr\"{a}sentierbare Moduln},  Archiv Der Mathematik {\bf 20} (1969), no. 3,  262-266.

\bibitem{MD07}  L. X. Mao, N. Q. Ding,  {\it Relative copure injective and copure flat modules}, J. Pure Appl. Algebra {\bf 208} (2007), 635-646.

\bibitem{M85} E. Matlis,  {\it Commutative semi-coherent and semi-regular rings}, J. Algebra {\bf 95} (1985), 343-372.

\bibitem{p08}  K. R. Pinzon, {\it Absolutely pure covers}, Commun. Algebra {\bf 36} (2008), 2186-2194.


\bibitem{S70} B. Stenstr\"{o}m,  {\it Coherent rings and FP-injective modules},  J. London Math. Soc.  {\bf 2} (1970), 323-329.

\bibitem{S79} B. Stenstr\"{o}m,   {\it  Rings of quotients}, Die Grundlehren Der Mathematischen Wissenschaften, Berlin: Springer-verlag, 1975.











\bibitem{fk16}  F. G. Wang,  H. Kim,  {\it  Foundations of Commutative rings and Their Modules}, Singapore: Springer, 2016.

\bibitem{WK21}  F. G. Wang,  H. Kim,   {\it  Relative FP-injective modules and relative IF rings}, Commun. Algebra  (2021), accepted, DOI: 10.1080/00927872.2021.1900861.





\bibitem{w} R. Wisbauer,   {\it  Foundations of Module and Ring Theory}, Algebra, Logic and Applications, Vol {\bf 3}, Amsterdam: Gordon and Breach, 1991.



\bibitem{ZWQ20} X. L. Zhang,   F. G. Wang,  W. Qi, {\it  On Characterizations of $w$-Coherent rings}, Commun. Algebra  {\bf 48}  (2020), no. 11, 4681-4697.


\bibitem{ZxlQ21} X. L. Zhang,  W. Qi, {\it Some Remarks on  $\phi$-Dedekind rings and $\phi$-\Prufer\ rings}, arxiv:2103.08278v1.

\bibitem{ZxlZ20} X. L. Zhang,  W. Zhao, {\it On $w$-$\phi$-flat modules and their homological dimensions},   Bull. Korean Math. Soc.,  accepted.

\bibitem{Z18} W. Zhao, {\it On $\phi$-flat modules and $\phi$-\Prufer\ rings}, J. Korean Math. Soc.  {\bf 55} (2018), no. 5, 1221-1233.

\bibitem{ZWZ20}  W. Zhao, F. G. Wang, X. L. Zhang, {\it On $\phi$-projective modules and $\phi$-\Prufer\ rings}, Commun. Algebra  {\bf 48} (2020), no. 7, 3079-3090.

\bibitem{ZWT13} W. Zhao, F. G. Wang, G. H.  Tang,  {\it On $\phi$-von Neumann regular rings},  J. Korean Math. Soc. \textbf{50} (2013), no. 1, 219-229.
\bibitem{ZZ19} W. Zhao, X. L. Zhang, {\it On Nonnil-injective modules}, J. Sichuan Normal Univ.  \textbf{42} (2019), no. 6, 808-815.
\end{thebibliography}
\end{document}